\documentclass[12pt,a4paper,english,reqno]{amsart}
\usepackage[a4paper,footskip=1.5em]{geometry}
\usepackage{amsmath,amssymb,amsthm,mathtools}
\usepackage[mathscr]{euscript}
\usepackage[usenames,dvipsnames]{color}
\usepackage{adjustbox,tikz,calc,graphics,babel,standalone}
\usepackage{subcaption}
\usepackage{csquotes,enumerate,verbatim}
\usepackage[final]{microtype}
\usepackage[numbers]{natbib}
\usetikzlibrary{shapes.misc,calc,intersections,patterns,decorations.pathreplacing}
\usepackage{hyperref}
\hypersetup{colorlinks=true,linkcolor=blue,citecolor=blue,pdfpagemode=UseNone,pdfstartview={XYZ null null 1.00}}
\usepackage{cmtiup}

\theoremstyle{plain}
\newtheorem*{theorem*}{Theorem}
\newtheorem{theorem}{Theorem}[section]
\newtheorem{lemma}[theorem]{Lemma}
\newtheorem{claim}[theorem]{Claim}
\newtheorem{proposition}[theorem]{Proposition}
\newtheorem*{claim*}{Claim}

\newtheorem{conjecture}[theorem]{Conjecture}

\theoremstyle{remark}

\def\N{\mathbb{N}}

\DeclareMathOperator\Deg{d}
\DeclareMathOperator\DegCo{\overline{d}}

\DeclareMathOperator{\EV}{\mathbb{E}}
\DeclareMathOperator{\PV}{\mathbb{P}}
\let\emptyset\varnothing
\let\eps\varepsilon

\let\originalleft\left
\let\originalright\right
\renewcommand{\left}{\mathopen{}\mathclose\bgroup\originalleft}
\renewcommand{\right}{\aftergroup\egroup\originalright}

\makeatletter
\def\imod#1{\allowbreak\mkern10mu({\operator@font mod}\,\,#1)}
\makeatother

\begin{document}

\title{Induced subgraphs with many distinct degrees}

\author{Bhargav Narayanan}
\address{Department of Pure Mathematics and Mathematical Statistics, University of Cambridge, Wilberforce Road, Cambridge CB3\thinspace0WB, UK}
\email{b.p.narayanan@dpmms.cam.ac.uk}

\author{Istv\'{a}n Tomon}
\address{Department of Pure Mathematics and Mathematical Statistics, University of Cambridge, Wilberforce Road, Cambridge CB3\thinspace0WB, UK}
\email{i.tomon@dpmms.cam.ac.uk}

\date{13 August 2016}

\begin{abstract}
Let $\hom(G)$ denote the size of the largest clique or independent set of a graph $G$. In 2007, Bukh and Sudakov proved that every $n$-vertex graph $G$ with $\hom(G) = O(\log n)$ contains an induced subgraph with $\Omega(n^{1/2})$ distinct degrees, and raised the question of deciding whether an analogous result holds for every $n$-vertex graph $G$ with $\hom(G) = O(n^\eps)$, where $\eps > 0$ is a fixed constant. Here, we answer their question in the affirmative and show that every graph $G$ on $n$ vertices contains an induced subgraph with $\Omega((n/\hom(G))^{1/2})$ distinct degrees. We also prove a stronger result for graphs with large cliques or independent sets and show, for any fixed $k \in \N$, that if an $n$-vertex graph $G$ contains no induced subgraph with $k$ distinct degrees, then $\hom(G) \ge n/(k-1)-o(n)$; this bound is essentially best-possible.
\end{abstract}

\maketitle

\section{Introduction}
A subset of the vertices of a graph is called \emph{homogeneous} if it induces either a clique or an independent set. What can one say about the structure of graphs with no large homogeneous sets? This is a central question in graph Ramsey theory and has received considerable attention over the last sixty years. Let $\hom(G)$ denote the size of the largest homogeneous set of a graph $G$. Erd\H{o}s and Szekeres~\citep{largeangle} proved that $\hom(G) \ge (\log n)/2$ for any graph $G$ on $n$ vertices and subsequently, Erd\H{o}s~\citep{upperramsey} used probabilistic arguments to demonstrate the existence of an $n$-vertex graph $G$ with $\hom(G) \le 2\log n$; here, and throughout the paper, all logarithms are base 2. Despite considerable effort, see~\citep{grol, bip_const, Cohen, Chatto} for instance, we know of no deterministic constructions of graphs with no large homogeneous sets; this suggests that such graphs should perhaps `look like' random graphs. This belief is supported by many results which show that graphs with no large homogeneous sets possess many of the same properties as random graphs.

For a constant $C>0$, we say that an $n$-vertex graph $G$ is \emph{$C$-Ramsey} if $\hom(G) \le C\log n$. There are a number of results which show that Ramsey graphs share various properties with dense random graphs. For example, Erd\H{o}s and Szemer\'edi~\citep{density} proved that Ramsey graphs must have edge densities bounded away from $0$ and $1$. Pr\"{o}mel and R\"{o}dl~\citep{universal} later showed that every $C$-Ramsey graph on $n$ vertices is $(\delta\log n)$-universal for some positive constant $\delta = \delta(C)$; here, a graph is \emph{$k$-universal} if it contains an induced copy of every graph on at most $k$ vertices. As a final example, let us mention that Shelah~\citep{shelah} proved that Ramsey graphs contain exponentially many non-isomorphic induced subgraphs.

The results of this paper are motivated by a line of questioning proposed by Bukh and Sudakov. Bukh and Sudakov~\citep{diffdegrees} proved, settling a conjecture of Erd\H{o}s, Faudree and S\'os~\citep{deg_ques}, that every $C$-Ramsey graph on $n$ vertices contains an induced subgraph with at least $\delta n^{1/2}$ distinct degrees for some constant $\delta(C) > 0$. They then raised the possibility of a similar result holding for graphs with much larger homogeneous sets, and suggested in particular that for any $0 <\eps <1/2$, every $n$-vertex graph $G$ with $\hom(G)<n^{\eps}$ should contain an induced subgraph with $\Omega(n^{1/2-\eps})$ distinct degrees; by building on their work, we establish this fact as a special case of Theorem~\ref{mainthm1} below. 

For a graph $G$, let $f(G)$ denote the largest integer $k$ for which $G$ contains an induced subgraph with $k$ distinct degrees. Our first result is the following.

\begin{theorem}\label{mainthm1}
For every graph $G$ on $n$ vertices, we have
\[f(G) \ge \frac{1}{250}\left(\frac{n}{\hom(G)}\right)^{1/2}.\]
\end{theorem}

We believe that this result is far from sharp however; for example, we are unable to construct an $n$-vertex graph $G$ either with $\hom(G)<n^{1/2}$ and $f(G) = o(n^{1/2})$, or with $\hom(G)\ge n^{1/2}$ and $f(G) = o(n/\hom(G))$. 

Our next result sharpens Theorem~\ref{mainthm1} for graphs with very large homogeneous sets. For a positive integer $k \in \N$, the disjoint union of $n/(k-1)$ cliques each of size $k-1$ gives us an example of a graph $G$ on $n$ vertices with $f(G) = k-1$ and $\hom(G) = \max\{n/(k-1), k-1\}$. We prove that if $k$ is fixed and $n$ is sufficiently large, then such a construction is essentially best-possible.

 \begin{theorem}\label{mainthm2}
Fix $k \in \N$ and $\eps>0$. If $n$ is sufficiently large, then $f(G) \ge k$ for every $n$-vertex graph $G$ with $\hom(G)< n/(k-1+\eps)$.
 \end{theorem}

We remark that with $\eps$ fixed, the minimal $n$ for which we are able to verify Theorem~\ref{mainthm2} is exponential in $k$. We further believe that it should be possible to prove Theorem~\ref{mainthm2} without an $\eps$-dependent error term; however, we are unable to do this at present.

This paper is organised as follows. In the next section, we establish some notation and collect together a few basic tools. In Section~\ref{section2}, we introduce the main ideas used in this paper and prove Theorem~\ref{mainthm1}. In Section~\ref{section3}, we refine the ideas used to prove Theorem~\ref{mainthm1} and prove Theorem~\ref{mainthm2}. We conclude with a discussion of some open problems in Section~\ref{conc}.

For the sake of clarity of presentation, we systematically omit floor and ceiling signs whenever they are not crucial. We also make no attempt to optimise the absolute constants in our results.

\section{Preliminaries}\label{prelim}
In this short section, we introduce some notation and collect together some facts that we shall rely on repeatedly in the sequel.

\subsection{Notation}
For us, a pair $\{x,y\}$ will always mean an unordered pair with $x \ne y$. For a set $X$, we write $X^{(2)}$ for the family of all pairs on the ground set $X$.

Let $G = (V,E)$ be a graph. We write $v(G)$ and $e(G)$ respectively for the number of vertices and edges of $G$. The complement of $G$ is denoted by $\overline{G}$. The neighbourhood of a vertex $x$ is denoted by $\Gamma(x)$, and the non-neighbourhood of $x$ is denoted by $\overline{\Gamma}(x)$. Also, let $\Deg(x)=|\Gamma(x)|$ denote the degree of $x$ in $G$, and let $\DegCo(x)=|\overline{\Gamma}(x)|$ be the degree of $x$ in $\overline{G}$.

For a subset $U \subset V$, we write $G[U]$ for the subgraph of $G$ induced by $U$. For a vertex $x \in V$, we write $\Gamma_{U}(x)$ for the set $\Gamma(x)\cap U$ and $\overline{\Gamma}_{U}(x)$ for the set $\overline{\Gamma}(x)\cap U$; we also set $\Deg_{U}(x)=|\Gamma_{U}(x)|$ and $\DegCo_{U}(x)=|\overline{\Gamma}_{U}(x)|$.

Finally, we define the \emph{neighbourhood-distance} between two vertices $x,y\in V$ by
\[\delta(x,y)=|(\Gamma(x)\setminus \{y\})\triangle (\Gamma(y)\setminus\{x\})|.\]
It is not hard to check that this distance satisfies the triangle inequality, i.e., $\delta(x,y)+\delta(y,z)\ge \delta(x,z)$ for any three vertices $x,y,z\in V$.

\subsection{Graph theoretic estimates}
We need the following simple fact about the neighbourhood-distance.
\begin{proposition}\label{boundeddegree}
For each $K\in\N$, there exists a $\Delta\in\N$ such that the following holds. If $G$ is a graph with $\delta(x,y)\le K$ for all $x,y\in V(G)$, then either $G$ or $\overline{G}$ has maximum degree at most $\Delta$.
\end{proposition}

\begin{proof}
We prove the claim with $\Delta=4K$. The proposition is trivial if $v(G) \le 4K$, so we may assume that $v(G) \ge 4K + 1$. Fix a vertex $v\in V(G)$ and let $s$ and $t$ be the number of edges and non-edges between $\Gamma(v)$ and $\overline{\Gamma}(v)$ respectively.

If $x\in \Gamma(v)$, then $\Deg_{\overline{\Gamma}(v)}(x)\le\delta(v,x)\le K$. Hence, $s\le K|\Gamma(v)|$ and analogously, $t\le K|\overline{\Gamma}(v)|$. Since $s+t=|\Gamma(v)||\overline{\Gamma}(v)|$, it follows that
\[ |\Gamma(v)||\overline{\Gamma}(v)|\le K(|\Gamma(v)|+|\overline{\Gamma}(v)|) = K(v(G)  - 1). \]
Since $v(G) \ge 4K+1$, this is only possible if one of $|\Gamma(v)|$ or $|\overline{\Gamma}(v)|$ is at most $2K$. If $|\Gamma(v)|\le 2K$, then it is not hard to see that every vertex has degree at most $3K$; indeed, this follows from the trivial observation that $\Deg(x) \le \delta(x,v) + \Deg(v)$ for any $x \in V(G)$. If $|\overline{\Gamma}(v)|\le 2K$ on the other hand,  then it is clear that the maximum degree of $\overline{G}$ is at most $3K$.
\end{proof}

We shall also require the following fact.

\begin{proposition}\label{smalldegree}
For any $k,\Delta \in \N$, there exists an $L\in\N$ such that the following holds. If $G$ is a graph with $f(G) < k$ and maximum degree at most $\Delta$, then $G$ contains at most $L$ vertices of degree at least $k-1$.
\end{proposition}
\begin{proof} We shall prove the claim with $L=(\Delta^{2}+1)k$. We say that two vertices $x,y \in V(G)$ are \emph{independent} if $(\Gamma(x)\cup \{x\})\cap (\Gamma(y)\cup \{y\})=\emptyset$. As the maximum degree of $G$ is at most $\Delta$, a vertex of $G$ is dependent on at most $\Delta + \Delta(\Delta-1) = \Delta^2$ other vertices. Let $S\subset V(G)$ be the set of vertices of degree at least $k-1$ and suppose for the sake of contradiction that $|S| > L$. Since $|S|/(\Delta^{2}+1) > k$, we can select $k$ pairwise independent vertices from $S$; let $x_{1},x_2,\dots,x_{k}$ be these vertices. For $1 \le i \le k$, let $X_i$ be an arbitrary $(i-1)$-element subset of $\Gamma(x_{i})$ and let
\[X =\{x_{1}, x_2, \dots,x_{k}\}\cup X_{1} \cup X_2 \cup \dots \cup X_k.\]
Since the vertices $x_{1}, x_2,\dots,x_{k}$ are pairwise independent, we see that $\Deg_{X}(x_{i})=i-1$ for each $1 \le i \le k$. It follows that $x_{1},x_2,\dots,x_{k}$ have different degrees in $G[X]$; this contradicts the fact that $f(G) < k$.
\end{proof}

Finally, we need the Caro--Wei theorem~\citep{caro, wei} which refines Tur\'{a}n's theorem.
\begin{proposition}\label{iTuran}
Every graph $G$ contains an independent set of size at least
\[ \sum_{v\in V(G)} \frac{1}{\Deg(v) + 1} \ge \frac{v(G)^2}{2e(G)+v(G)}.\]
\end{proposition}
\begin{proof}
Order the vertices of $G$ uniformly at random and consider the set $I$ of those vertices that precede all their neighbours in this ordering. Clearly, $I$ is independent; the proposition follows since
\[\EV[|I|] = \sum_{v\in V(G)} \frac{1}{\Deg(v) + 1}. \qedhere\]
\end{proof}

\subsection{Probabilistic inequalities}
We need the following two well-known inequalities; the proofs of these claims may be found in~\citep{prob_book}. We shall require Markov's inequality.
\begin{proposition}\label{iMarkov}
Let $X$ be a non-negative real-valued random variable. For any $t\ge 0$, we have
\[\PV(X>t)<\frac{\EV [X]}{t}.\eqno\qed\]
\end{proposition}

We shall also require Hoeffding's inequality.
\begin{proposition}\label{iHoeffding}
Let $X=\sum_{i=1}^{n}X_{i}$ where $X_{1},X_2,\dots,X_{n}$ are independent real-valued random variables with $0\le X_{i}\le 1$ for each $1\le i\le n$. For any $t \ge 0$, we have
\[\PV(|X-\EV[X]|\ge t)\le 2\exp\left(\frac{-2t^{2}}{n}\right). \eqno\qed\]
\end{proposition}

\section{Small homogeneous sets}\label{section2}
This section is devoted to the proof of Theorem~\ref{mainthm1}. Before we prove the result, let us give an overview of the proof. To show that a graph $G$ contains an induced subgraph with many distinct degrees, we pick a random subset $U$ of the vertices and consider the subgraph of $G$ induced by $U$. We first show that if the neighbourhood-distance between a pair of vertices $x,y\in U$ is large, then the probability of $\Deg_{U}(x)$ and $\Deg_{U}(y)$ being equal is small. Next, to show that there are many such pairs of vertices with large neighbourhood-distances, we bound the number of pairs of vertices with small neighbourhood-distances in terms of $\hom(G)$. We now make this sketch precise.

\begin{proof}[Proof of Theorem~\ref{mainthm1}] The result is trivial if our graph has fewer than 250 vertices, so we may assume that $G=(V,E)$ is a graph on $n \ge 250$ vertices. Let $U$ be a random subset of $V$ obtained by selecting each vertex of $G$ with probability $1/2$, independently of the other vertices. Define an auxiliary \emph{degree graph} $D$ on $U$ where two vertices $x,y\in U$ are joined by an edge if they have the same degree in $G[U]$; note that $D$ is the union of vertex disjoint cliques, one for each vertex degree in $G[U]$. For any pair of distinct vertices $x,y\in V$, the probability that $xy$ is an edge of $D$ is precisely $\PV(\Deg_{U}(x)=\Deg_{U}(y))/4$. Hence, the expected number of edges of $D$ is given by
\[\EV[e(D)]=\frac{1}{4}\sum_{\{x,y\} \in V^{(2)}}\PV(\Deg_{U}(x)=\Deg_{U}(y)).\]

We shall bound the probability of the event $\Deg_{U}(x)=\Deg_{U}(y)$ by the neighbourhood-distance between $x$ and $y$; a similar result appears in~\cite{diffdegrees}.
\begin{lemma}\label{distinctprob}
For any pair of distinct vertices $x,y\in V$, we have
\[\PV(\Deg_{U}(x)=\Deg_{U}(y))<\frac{20}{\sqrt{\delta(x,y)+1}}.\]
\end{lemma}
\begin{proof} Let $s=|\Gamma(x)\setminus \Gamma(y)|$ and $t=|\Gamma(y)\setminus \Gamma(x)|$. Observe that $s+t$ is equal to $\delta(x,y)+2$ if $xy\in E$ and $\delta(x,y)$ otherwise. Without loss of generality, we may suppose that $s\ge t$. Therefore, $s\ge \delta(x,y)/2$ and it follows that
\begin{align*}
\PV(\Deg_{U}(x)=\Deg_{U}(y))&=\frac{1}{2^{s+t}}\sum_{i=0}^{t}\binom{s}{i}\binom{t}{i}\\
&\le \frac{1}{2^{s}}\max_{0\le i\le s}\binom{s}{i}\left(\frac{1}{2^{t}}\sum_{i=0}^{t}\binom{t}{i}\right)\\
&=\frac{1}{2^{s}}\binom{s}{\lfloor s/2\rfloor} <\frac{10}{\sqrt{s+1}} \le \frac{20}{\sqrt{\delta(x,y)+1}}.\qedhere
\end{align*}
\end{proof}

As an immediate consequence of this lemma, we have
\begin{equation}\label{equ1}
\EV [e(D)]<\sum_{\{x,y\} \in V^{(2)}}\frac{5}{\sqrt{\delta(x,y)+1}}.
\end{equation}
We now bound the right hand side of~\eqref{equ1} in terms of $\hom(G)$. Fix a vertex $x\in V$ and consider the sets $S=\Gamma(x)$ and $T=\overline{\Gamma}(x)$. Note that $\delta(x,y) \ge \Deg_{T}(y)$ for any $y\in T$. By applying Proposition~\ref{iTuran} to the graph $G[T]$, we get
\[\hom(G)\ge\sum_{y\in T}\frac{1}{\Deg_{T}(y)+1}\ge \sum_{y\in T}\frac{1}{\delta(x,y)+1}.\]
Similarly, if $y\in S$, then $\delta(x,y) \ge \DegCo_{S}(y)$. Applying Proposition~\ref{iTuran} to the complement of $G[S]$, we similarly get
\[\hom(G)\ge\sum_{y\in S}\frac{1}{\DegCo_{S}(y)+1}\ge \sum_{y\in S}\frac{1}{\delta(x,y)+1}.\]
It follows that
\[2\hom(G)\ge\sum_{y\in V\setminus \{x\}} \frac{1}{\delta(x,y)+1}.\]
Finally, by summing this inequality over all the vertices of $G$ and applying the Cauchy--Schwarz inequality, we get
\begin{equation}\label{equ2}
n\hom(G)\ge \sum_{\{x,y\} \in V^{(2)}} \frac{1}{\delta(x,y)+1} \ge \binom{n}{2}^{-1}\left(\sum_{\{x,y\} \in V^{(2)}} \frac{1}{\sqrt{\delta(x,y)+1}}\right)^{2}.
\end{equation}

From~\eqref{equ1} and~\eqref{equ2}, we conclude that
\[\EV[e(D)]<4\sqrt{n^3\hom(G)}.\]
It follows from Markov's inequality that
\[\PV\left( e(D)>12\sqrt{n^3\hom(G)}\right)<\frac{1}{3}.\]
Next, since $|U|$ is the sum of $n \ge 250$ independent indicator random variables and $\EV[|U|]=n/2$, it follows from Hoeffding's inequality that
\[\PV(|U|<n/3)<\frac{1}{3}.\]
Hence, with positive probability, the degree graph $D$ satisfies $v(D)>n/3$ and $e(D)<12(n^3\hom(G))^{1/2}$. Applying Proposition~\ref{iTuran} to $D$, we see that this graph contains an independent set of size at least
\[\frac{v(D)^{2}}{2e(D)+v(D)}\ge\frac{1}{250}\sqrt{\frac{n}{\hom(G)}}.\]
The vertices of this independent set have different degrees in $G[U]$ by the definition of $D$. Therefore, the subgraph induced by $U$ has at least $(1/250)(n/\hom(G))^{1/2}$ distinct degrees with positive probability; the result follows.
\end{proof}

 \section{Large homogeneous sets}\label{section3}
In this section, we consider graphs with large homogeneous sets and prove Theorem~\ref{mainthm2}. Before we turn to the proof of Theorem~\ref{mainthm2}, let us identify the inefficiencies in the proof of Theorem~\ref{mainthm1}; to this end, we consider some examples of graphs with large homogeneous sets that contain no induced subgraphs with many distinct degrees. Let $b,k,n\in\N$ be positive integers satisfying $b \le k\le (\log n)/2$ and consider
\begin{enumerate}
\item\label{ex1} a disjoint union of $n/k$ cliques each of size $k$,
\item\label{ex2} a disjoint union of $k$ cliques each of size $n/k$, and
\item\label{ex3} a disjoint union of $k/b$ copies of $\overline{H}$, where $H$ is the disjoint union of $n/k$ cliques each of size $b$.
\end{enumerate}
It is not hard to see that if $G$ is the graph in either of the above three examples, then $\hom(G)=n/k$ and $f(G) = k$. Suppose that $G$ is one of the graphs described above; let $U$ be a $1/2$-random subset of $V(G)$ and define the degree graph $D$ on $U$ as in the proof of Theorem~\ref{mainthm1}.

If $G$ is the graph in example~\eqref{ex1}, then with high probability, $G[U]$ has about
\[\frac{n}{2^{k}}\binom{k-1}{i}\] vertices of degree $i$ for each $0 \le i \le k-1$. This means that while $G[U]$ has $k$ distinct degrees, the $k$ cliques in the auxiliary graph $D$ have very different sizes. Hence, the final application of the Caro--Wei bound in the proof of Theorem~\ref{mainthm1} is too crude for this graph.

If $G$ is the graph in example~\eqref{ex2}, then $G[U]$ again has $k$ distinct degrees with high probability. However, notice that $\delta(x,y)$ is either $0$ or $2n/k-2$ for any pair of vertices $x,y \in V(G)$ with the former being the case for about $1/k$ of the pairs. Our application of the Cauchy--Schwartz inequality in the proof of Theorem~\ref{mainthm1} is therefore suboptimal for this graph.

Our third example is a generalisation of the first two examples. The graph defined in example~\eqref{ex3} is the complement of the graph in our first example if $b=k$, and is the graph in our second example when $b=1$. When $1<b<k$, this example illustrates both the aforementioned inefficiencies in the proof of Theorem~\ref{mainthm1}.

Nevertheless, notice that in all of our examples, a random induced subgraph has at least $k$ distinct degrees with high probability; our strategy for proving Theorem~\ref{mainthm2} will be informed by this fact. Let us now sketch our strategy. First, we show that we may partition the vertex set of our graph $G$ into groups in such a way that vertices within a group are all close together in neighbourhood-distance while vertices from different groups are far apart. We then select a random subset $U$ of the vertices as before. We use the bound on $\hom(G)$ to show that in $G[U]$, the degrees of the vertices within a group are distributed like the vertex degrees in a random induced subgraph of the graph in example~\eqref{ex1}. We then argue that most pairs of vertices from different groups have different degrees in $G[U]$. We finally deduce from these facts that the number of distinct degrees in $G[U]$ is large.

We are ready to prove Theorem~\ref{mainthm2}; while following the proof, the reader is encouraged to keep example~\eqref{ex3} in mind.

\begin{proof}[Proof of Theorem~\ref{mainthm2}] The result is trivial when $k = 1$, so fix $k \ge 2$. Also, we may assume without loss of generality that $0 < \eps < 1/2$. With the benefit of hindsight, we define the constants $\beta = \eps/10k$, $\eta = \eps \beta/10^{5}k^{2}  $ and $J = 10^4 k^{12} 2^{4k} /\eta^{4}$.

Assume that $n$ is sufficiently large and suppose for the sake of contradiction that $G=(V,E)$ is a graph on $n$ vertices with $\hom(G) < n/(k-1 + \eps)$ and $f(G) < k$.

Let $U$ be a random subset of $V$ obtained by selecting each vertex of $V$ with probability $1/2$, independently of the other vertices. As before, define an auxiliary degree graph $D$ on $U$ where two vertices $x,y\in U$ are joined by an edge if they have the same degree in $G[U]$.

For a set of vertices $W\subset V$, we define
\[{\hat \delta}(W)=\sum_{\{x,y\} \in W^{(2)}}\frac{5}{\sqrt{\delta(x,y)+1}}.\]
We first prove the following.

\begin{claim}\label{smalldistance}
For each $W\subset V$, we have
\[{\hat \delta}(W)>\frac{|W|^{2}-3k|W|}{54k}.\]
\end{claim}

\begin{proof} The statement is trivial if $|W| \le 3k$, so suppose that $|W| \ge 3k+1 \ge 7$. Note that by Lemma~\ref{distinctprob}, the quantity ${\hat \delta}(W)$ is an upper bound for the expected number of edges of $D$ spanned by $W\cap U$; in other words,
\[ \EV[e(D[W\cap U])]\le {\hat \delta}(W).\]
Applying Markov's inequality, we get
\[\PV(e(D[W\cap U])>3{\hat \delta}(W))<1/3.\]
Also, $|W\cap U|$ is the sum of $|W| \ge 7$ independent indicator random variables, so by Hoeffding's inequality,
\[\PV(|W\cap U|<|W|/3)\le 2\exp(-14/9) < 1/2.\]
Hence, with positive probability, we have both  $e(D[W\cap U])\le3{\hat \delta}(W)$ and $|W\cap U|\ge |W|/3$. By Proposition~\ref{iTuran}, with positive probability, the graph $D[W\cap U]$ contains an independent set of size at least
\[\frac{|W\cap U|^{2}}{2e(D[W\cap U])+|W\cap U|}\ge \frac{|W|^{2}}{54{\hat \delta}(W)+3|W|}.\]
However, since $f(G) < k$, we know that $D$ cannot contain an independent set of size $k$. It follows that
\[k>\frac{|W|^{2}}{54{\hat \delta}(W)+3|W|},\]
or equivalently,
\[{\hat \delta}(W)>\frac{|W|^{2}-3k|W|}{54k}. \qedhere\]
\end{proof}

This tells us that every large subset of $V$ must have many pairs of vertices whose neighbourhood-distance is small. This suggests that we should be able to group the vertices of $G$ into large groups in such a way that the neighbourhood-distance between any two vertices within a group is small. We prove such a statement this next; recall that $\beta = \eps/10k$ and $J = 10^{24} k^{20} 2^{4k} /(\eps\beta)^{4}$.

\begin{claim}\label{partitionclaim}
There exists a $K=K(k,\eps) > 0$ such that the following holds. There are pairwise disjoint sets $A_{1},A_2,\dots,A_{m}\subset V$ such that
\begin{enumerate}
\item $|A_1 \cup A_2 \cup \dots \cup A_m|>(1-\beta)n$,
\item $|A_{i}| > \beta n / 10^{4}k$ for each $1 \le i \le m$,
\item if $x,y\in A_{i}$, then $\delta(x,y)<K$, and
\item if $x\in A_{i}$ and $y\in A_{j}$ with $i \ne j$, then $\delta(x,y)>J$.
\end{enumerate}
\end{claim}
\begin{proof}
We prove the claim with
\[ K=2 \left( 10^{6}k^{2}+\frac{J\log (10^{4}k)}{\log(1+\beta/2)}\right).\]
We shall construct a sequence of pairwise disjoint sets $A_{1},A_{2},\dots$ and another nested sequence of sets $B_{0}\subset B_{1}\subset\dots$ such that for each $l \in \N$,
\begin{enumerate}[(i)]
\item $B_l$ is disjoint from $A_1 \cup A_2 \cup \dots \cup A_l$,
\item $|A_{l}|>\beta n / 10^{4}k$,
\item $|B_{l}|<(\beta/2)\sum_{i=1}^{l}|A_{i}|$,
\item if $x,y\in A_{l}$, then $\delta(x,y)<K$, and
\item if $x\in A_{l}$ and $y\in V\setminus (A_{l}\cup B_{l})$, then $\delta(x,y)>J$.
\end{enumerate}
We set $B_{0}=\emptyset$. Let $l\ge 0$ and suppose that the sets $A_{1},A_2,\dots,A_{l}$ and $B_{l}$ have already been constructed satisfying the above properties. The claim is proved if $\sum_{i=1}^{l}|A_{i}|>(1-\beta)n$. Suppose otherwise; we construct $A_{l+1}$ and $B_{l+1}$ as follows.

Define $W=V\setminus (B_{l}\cup A_1 \cup A_2\cup \dots \cup A_l)$ and note that $|W|>\beta n/2$. Let $p$ be the number of pairs $\{x,y\}\in W^{(2)}$ such that $\delta(x,y)<10^{6}k^{2}$. Note that we have
\[{\hat \delta}(W)=\sum_{\{x,y\}\in W^{(2)}}\frac{5}{\sqrt{\delta(x,y)+1}}<5p+\frac{|W|^{2}}{10^{3}k}.\]
On the other hand, if $n$ is sufficiently large, then by Claim~\ref{smalldistance}, we have
\[{\hat \delta}(W)>\frac{|W|^{2}-3k|W|}{54k}>\frac{|W|^{2}}{100k}.\]
It follows that $p>|W|^{2} / 10^{3}k$. Thus, there is a vertex $w \in W$ and a subset $S \subset W$ with $|S|>|W|/10^{3}k$ where every $x\in S$ satisfies $\delta(w,x)<10^{6}k^{2}$.

We set $C_{1}=\{w\}\cup S$ and iteratively construct an increasing sequence of sets $C_{1}\subset C_{2}\subset\dots$ as follows. Having constructed $C_{i}$, consider the set $T$ of vertices $x\in W\setminus C_{i}$ for which there exists a vertex $y\in C_{i}$ such that $\delta(x,y)\le J$. If $|T|\ge\beta |C_{i}|/2$, then set $C_{i+1}=C_{i}\cup T$. Otherwise, stop and define $A_{l+1}=C_{i}$ and $B_{l+1}=B_{l}\cup T$.

It is clear that $B_{l+1} \cap (A_1 \cup A_2 \cup \dots \cup A_{l+1}) = \emptyset$ and that  $\delta(x,y) > J$ for all $x\in A_{l+1}$ and $y\in V\setminus (B_{l+1}\cup A_{l+1})$. It is also clear that
\[|A_{l+1}|\ge|C_{1}|>\frac{|W|}{10^{3}k} >\frac{\beta n}{10^{4}k}.\]
Finally, note that we must have $|T|<\beta|C_i|/2$ when we define $A_{l+1}$ and $B_{l+1}$, so we inductively have
\[ |B_{l+1}|< \frac{\beta}{2}\sum_{i=1}^{l+1}|A_{i}|.\]

Next, observe that for each $i \ge 1$, we have
\[|C_{i}|\ge \left(1+\frac{\beta}{2}\right)^{i-1}|C_{1}|\ge \left(1+\frac{\beta}{2}\right)^{i-1}\frac{|W|}{10^{3}k}.\] 
As each $C_{i}$ is a subset of $W$, it is clear that the sets $A_{l+1}$ and $B_{l+1}$ get defined after at most $\log (10^{4}k)/\log (1+\beta/2)$ iterations. The neighbourhood-distance satisfies the triangle inequality; consequently,
\[\delta(w,x)< 10^{6}k^{2}+(i-1)J\]
for all $x\in C_{i}$ and therefore,
\[\delta(x,y) < 2(10^{6}k^{2}+(i-1)J)\]
for all $x,y\in C_{i}$.
Thus, for all $x,y\in A_{l+1}$, we have
\[\delta(x,y) < 2 \left(10^{6}k^{2}+\frac{J\log (10^{4}k)}{\log(1+\beta/2)}\right) = K.\]
Therefore, the sets $A_{1}, A_2, \dots,A_{l+1}$ and $B_{l+1}$ also satisfy the properties described above.

To finish the proof, note that since the sets $A_1, A_2,\dots$ are pairwise disjoint, the above described construction procedure must terminate; indeed, if $m > 10^{4}k/\beta$, then $\sum_{i=1}^{m}|A_{i}|>(1-\beta)n$.
\end{proof}

Let $K$ and $A_{1},A_2,\dots,A_{m}$ be as promised by Claim~\ref{partitionclaim}. Note that $m \le 10^{4}k/\beta$ since $|A_{i}| > \beta n/10^{4}k$ for each $1 \le i \le m$.

To proceed, we need to introduce the following notion of independence. If $S\subset V$ and $x,y\in S$, we say that $x$ and $y$ are \emph{$S$-independent in $G$} if
\[(\Gamma_{S}(x)\cup\{x\})\cap (\Gamma_{S}(y)\cup\{y\})=\emptyset;\]
analogously, we say that $x$ and $y$ are \emph{$S$-independent in $\overline{G}$} if
\[(\overline\Gamma_{S}(x)\cup\{x\})\cap (\overline\Gamma_{S}(y)\cup\{y\})=\emptyset.\]
Recall that $U$ is a $1/2$-random subset of the vertices; therefore, if $x$ and $y$ are $S$-independent, then $\Deg_{U\cap S}(x)$ and $\Deg_{U\cap S}(y)$ are independent random variables.

For $1 \le i \le m$, let $r_{i}$ be the unique non-negative integer such that
\[\frac{r_{i}n}{k-1+\eps}<|A_{i} | - \eta n \le \frac{(r_{i}+1)n}{k-1+\eps}.\]
Our next step is to show that each $A_i$ has a large subset of pairwise $A_{i}$-independent vertices; recall that $\eta = \eps \beta/10^{5}k^{2}$.

\begin{claim}\label{independentvertices}
For each $1 \le i \le m$, there is a subset $B_i\subset A_{i}$ with $|B_i|>\eta n/2k^2$ such that either
\begin{enumerate}
\item $r_{i}\le \Deg_{A_i}(x)\le k-2$ for every $x\in B_i$, and
\item any pair of distinct vertices $x,y\in B_i$ are $A_{i}$-independent in $G$,
\end{enumerate}
or
\begin{enumerate}
\item $r_{i}\le \DegCo_{A_i}(x)\le k-2$ for every $x\in B_i$, and
\item any pair of distinct vertices $x,y\in B_i$ are $A_{i}$-independent in $\overline{G}$.
\end{enumerate}
\end{claim}
It is clear from the definition of $r_i$ that $r_i \le k-1$; the above claim implicitly tells us that our assumptions about $G$ actually imply that $r_i \le k-2$ for each $1 \le i \le m$.
\begin{proof}[Proof of Claim~\ref{independentvertices}]
Fix $1\le i\le m$ and let $r=r_i$, $A= A_i$ and $F = G[A_i]$. To avoid confusion, we write $\Deg_F$ and $\delta_F$ to denote the degrees and neighbourhood-distances in the graph $F$.

Note that $\delta_{F}(x,y)\le\delta(x,y) < K$ for all $x,y\in A$. Hence, by applying Proposition~\ref{boundeddegree} to $F$, we see that there exists a $\Delta=\Delta(k,\eps)$ such that either $\Deg_{F}(x) \le \Delta$ for all $x\in A$ or $\DegCo_{F}(x) \le \Delta$ for all $x\in A$. We suppose that the former holds; the other case may be handled analogously.

We now apply Proposition~\ref{smalldegree} to $F$ and conclude that all but at most $L=L(k,\eps)$ vertices $x\in A$ satisfy $\Deg_F(x) \le k-2$. Let $q$ be the number of vertices $x \in A$ with $\Deg_F(x) \le r-1$. Then, by Proposition~\ref{iTuran}, $F$ contains an independent set of size at least
$q/r.$ From our assumption about $\hom(G)$, it follows that
\[q\le r\hom(F) \le r\hom(G) <\frac{rn}{k-1+\eps};\]
therefore, there are at least $\eta n$ vertices $x\in A$ with $\Deg_F (x) \ge r$.

Let $S_1$ be the set of vertices $x \in A$ with $\Deg_{F}(x)\ge r$. Let $S_2$ be the set of those $x \in S_1$ which are $A$-dependent on some vertex $y\in A$ with $\Deg_F(y) \ge k-1$; the number of such $y$ is at most $L$, so it follows that $|S_2|\le L(1 + \Delta + \Delta(\Delta-1))$. If $n$ is sufficiently large, then the set $S_3 = S_1 \setminus S_2$ contains at least $2\eta n/3$ vertices. For every vertex $x\in S_3$, there are at most $(k-2) + (k-2)(k-3) \le k^2 - 1$ other vertices $y\in A$ such that $x$ and $y$ are $A$-dependent. Hence, we can greedily select a set $B \subset S_3$ of pairwise $A$-independent vertices of size at least $|S_3|/k^{2}> \eta n/2k^2$.
\end{proof}

For $1 \le i \le m$, let $B_{i}\subset A_{i}$ be as guaranteed by Claim~\ref{independentvertices} and let $\mathcal{A}_i$ be the event that there exists an integer $d_i \ge 0$ and $r_i + 1$ pairwise disjoint sets $B_{i,0}, B_{i,1}, \dots, B_{i,r_i} \subset B_i \cap U$ such that for all $0\le s,t \le r_i$,
\begin{enumerate}
\item $|B_{i,s}| \ge \eta n/2^{k+1}k^{3}$,
\item $\Deg_{U\cap A_{i}}(x) = \Deg_{U\cap A_{i}}(y)$ for all $x, y \in B_{i,s}$,
\item $\Deg_{U\cap A_{i}}(x) \ne \Deg_{U\cap A_{i}}(y)$ for all $x \in B_{i,s}$ and $y \in B_{i,t}$ with $s \ne t$, and
\item $\Deg_{U\cap(V\setminus A_{i})}(x)=d_{i}$ for all  $x \in B_{i,0} \cup B_{i,1} \cup \dots \cup B_{i,r_i}$.
\end{enumerate}
We next prove the following bound.
\begin{claim}
For each $1 \le i \le m$,
\[
\PV(\mathcal{A}_i) > 1-\frac{1}{3m}.
\] 
\end{claim}
\begin{proof}
Fix $1 \le i \le m$ and suppose that all pairs of distinct vertices from $B_i$ are $A_i$-independent in $G$; the other case may be handled similarly by working with vertex degrees in $\overline{G}$ instead of $G$.

For each $W \subset V \setminus A_i$, we shall prove that 
\[
\PV(\mathcal{A}_i\,|\,U \cap (V\setminus A_i) = W) > 1-\frac{1}{3m}.
\]
We fix $W \subset V \setminus A_i$ and condition on the event $U \cap (V\setminus A_i) = W$; note that any event that depends only on the vertices in $A_i$ is independent of this event. If the collection of degrees $\{\Deg_{W}(x)\}_{x\in B_{i}}$ contains $k$ distinct integers, then the subgraph induced by $B_{i}\cup  W$  has $k$ different degrees, contradicting our assumption that $f(G) < k$. Hence, there exists an integer $d_i\ge 0$ for which there are least $|B_{i}|/k$  vertices $x\in B_{i}$ with $\Deg_{W}(x)=d_{i}$; let $C_{i}\subset B_{i}$ be the set of these vertices. 

For a vertex $x\in C_{i}$ and $0\le s \le r_{i}$, let $I_i(x,s)$ be the indicator of the event $ \{x \in U\}\cap \{\Deg_{U\cap A_{i}}(x)=s\}  $ and define 
\[I_{i}(s)=\sum_{x\in C_{i}} I_i(x,s).\] 
Observe that $\EV [I_i(x,s) ]\ge 1/2^{k-1}$ since $r_{i}\le \Deg_{A_{i}}(x)\le k-2$ and consequently,
\[\EV [I_{i}(s)]\ge \frac{|C_{i}|}{2^{k-1}}\ge\frac{\eta n}{2^{k}k^{3}}.\]
Also, the indicator random variables $\{I_i(x,s)\}_{x\in C_{i}}$ are independent, so by Hoeffding's inequality, we have
\[\PV\left(I_{i}(s)<\frac{\eta n}{2^{k+1}k^{3}}\right)<\frac{1}{3mk}\]
for all sufficiently large $n$ since $m \le 10^{4}k/\beta$. If $I_i(s) \ge \eta n/2^{k+1}k^{3}$ for each $0 \le s \le r_i$, then $\mathcal{A}_i$ clearly holds. Since $r_i \le k-2$, the claim follows by the union bound.
\end{proof}

Define the event $\mathcal{A} = \mathcal{A}_1 \cap \mathcal{A}_2 \cap \dots \cap \mathcal{A}_m$; it is clear from the previous claim that $\PV(\mathcal{A})>2/3$.

Let us remind the reader that $D$ is the graph on $U$ where $x,y\in U$ are joined by an edge if $\Deg_{U}(x)=\Deg_{U}(y)$. Let $b$ be the number of edges $xy$ of $D$ with $x\in A_{i}$ and $y\in A_{j}$ for some $1\le i<j\le m$. As $\delta(x,y)> J$ for such a pair of vertices $x$ and $y$, it follows from Claim~\ref{distinctprob} that
\[\PV(xy\in E(D))<\frac{5}{\sqrt{J+1}}
\]
and hence, $\EV [b]<5n^{2}/J^{1/2}$. 

Let $\mathcal{B}$ be the event that $b \le 15n^{2}/J^{1/2}$. By Markov's inequality, $\PV(\mathcal{B}) > 2/3$. We finish the proof of Theorem~\ref{mainthm2} by proving the following.

\begin{claim}
If both $\mathcal{A}$ and $\mathcal{B}$ hold, then $G[U]$ has at least $k$ distinct degrees.
\end{claim}
\begin{proof}
We know that $D$ is the union of vertex disjoint cliques, one for each vertex degree in $G[U]$. If $\mathcal{A}$ holds, then for each $1\le i \le m$, the graph $D[A_{i}\cap U]$ contains at least $r_{i}+1$ disjoint cliques $D_{i,0},D_{i,1},\dots,D_{i,r_{i}}$ (which correspond to the degrees in $G[U]$ of the vertices in $B_{i,0}, B_{i,1},\dots,B_{i,r_{i}}$) each of size at least $\eta n/2^{k+1}k^{3}$. Additionally, if $\mathcal{B}$ holds, then there are no edges in $D$ between $D_{i,s}$ and $D_{j,t}$ for any $1\le i<j\le m$, $0\le s \le r_{i}$ and $0 \le t  \le r_{j}$. Indeed, if not, then $D_{i,s}\cup D_{j,t}$ induces a clique in $D$ and we arrive at a contradiction since we would then have
\[b\ge|D_{i,s}||D_{j,t}|\ge\frac{\eta^{2}n^{2}}{2^{2k+2}k^{6}}>\frac{15n^{2}}{\sqrt{J}};\]
here, the last inequality holds since $J = 10^4 2^{4k} k^{12}/ \eta^4$. Hence, for all $1 \le  i \le m$ and $0 \le s \le r_{i}$, the cliques $D_{i,s}$ all correspond to distinct degrees in $G[U]$. Thus, $G[U]$ has at least $\sum_{i=1}^{m}(r_{i}+1)$ different degrees. Since
\[ \frac{(r_{i}+1)n}{k-1+\eps} \ge |A_{i} | - \eta n ,\]
it follows that
\[\sum_{i=1}^{m}(r_{i}+1)\ge (1-\beta - m\eta)(k-1+\eps)>k-1;\]
here, the last inequality holds since $\beta = \eps/10k$, $\eta=\eps\beta/10^{5}k$ and $m \le 10^4k/\beta$.
Since $\sum_{i=1}^{m}(r_{i}+1)$ is an integer, this sum is at least $k$. We conclude that if both $\mathcal{A}$ and $\mathcal{B}$ hold, then $G[U]$ has at least $k$ distinct degrees.
\end{proof}

We know by the union bound that $\PV(\mathcal{A} \cap \mathcal{B}) > 0$. Therefore, $G[U]$ has $k$ distinct degrees with positive probability, contradicting our assumption that $f(G) < k$; the result follows.
\end{proof}

\section{Conclusion}\label{conc}
We conclude this note by discussing two of the questions we alluded to in the introduction. First, as we mentioned earlier, we suspect that the following strengthening of Theorem~\ref{mainthm2} is true.

\begin{conjecture} Fix a positive integer $k \ge 2$. If $n$ is sufficiently large, then $f(G) \ge k$ for every $n$-vertex graph $G$ with $\hom(G) < n/(k-1)$.
\end{conjecture}

It may be read out of the proof of Theorem~\ref{mainthm2} that if $G$ is an $n$-vertex graph with $\hom(G) > n/(100\log n)$, then $f(G) = \Omega(n/\hom(G))$; for such graphs, this is a significant improvement over Theorem~\ref{mainthm1}. We believe that it should be possible to prove a similar result for graphs with much smaller homogeneous sets.
\begin{conjecture} If $G$ is an $n$-vertex graph with $\hom(G) > n^{1/2}$, then $f(G) = \Omega(n/\hom(G))$.
\end{conjecture}

\bibliographystyle{amsplain}
\bibliography{distinct_degrees}

\providecommand{\bysame}{\leavevmode\hbox to3em{\hrulefill}\thinspace}
\providecommand{\MR}{\relax\ifhmode\unskip\space\fi MR }
\providecommand{\MRhref}[2]{%
  \href{http://www.ams.org/mathscinet-getitem?mr=#1}{#2}
}
\providecommand{\href}[2]{#2}
\begin{thebibliography}{10}

\bibitem{prob_book}
N.~Alon and J.~H. Spencer, \emph{The probabilistic method},
  3\textsuperscript{rd} ed., Wiley-Interscience Series in Discrete Mathematics
  and Optimization, John Wiley \& Sons, Inc., Hoboken, NJ, 2008.

\bibitem{bip_const}
B.~Barak, A.~Rao, T.~Shaltiel, and A.~Wigderson, \emph{2-source dispersers for
  {$n^{o(1)}$} entropy, and {R}amsey graphs beating the {F}rankl-{W}ilson
  construction}, Ann. of Math. \textbf{176} (2012), 1483--1543.

\bibitem{diffdegrees}
B.~Bukh and B.~Sudakov, \emph{Induced subgraphs of {R}amsey graphs with many
  distinct degrees}, J. Combin. Theory Ser. B \textbf{97} (2007), 612--619.

\bibitem{caro}
Y.~Caro, \emph{New results on the independence number}, Technical Report, Tel
  Aviv University (1979).

\bibitem{Chatto}
E.~Chattopadhyay and D.~Zuckerman, \emph{Explicit two-source extractors and
  resilient functions}, Proceedings of the {F}orty-eighth {A}nnual {ACM}
  {S}ymposium on {T}heory of {C}omputing, ACM, New York, NY, 2016,
  pp.~670--683.

\bibitem{Cohen}
G.~Cohen, \emph{Two-source dispersers for polylogarithmic entropy and improved
  ramsey graphs}, Proceedings of the {F}orty-eighth {A}nnual {ACM} {S}ymposium
  on {T}heory of {C}omputing, ACM, New York, NY, 2016, pp.~278--284.

\bibitem{upperramsey}
P.~Erd\H{o}s, \emph{Some remarks on the theory of graphs}, Bull. Amer. Math.
  Soc. \textbf{53} (1947), 292--294.

\bibitem{largeangle}
P.~Erd\H{o}s and G.~Szekeres, \emph{A combinatorial problem in geometry},
  Compositio Math. \textbf{2} (1935), 463--470.

\bibitem{deg_ques}
P.~Erd{\H{o}}s, \emph{Some of my favourite problems in various branches of
  combinatorics}, Matematiche (Catania) \textbf{47} (1992), 231--240.

\bibitem{density}
P.~Erd{\H{o}}s and A.~Szemer{\'e}di, \emph{On a {R}amsey type theorem}, Period.
  Math. Hungar. \textbf{2} (1972), 295--299.

\bibitem{grol}
V.~Grolmusz, \emph{Low rank co-diagonal matrices and {R}amsey graphs},
  Electron. J. Combin. \textbf{7} (2000), Research Paper 15.

\bibitem{universal}
H.~J. Pr{\"o}mel and V.~R{\"o}dl, \emph{Non-{R}amsey graphs are {$c\log
  n$}-universal}, J. Combin. Theory Ser. A \textbf{88} (1999), 379--384.

\bibitem{shelah}
S.~Shelah, \emph{Erd{\H o}s and {R}\'enyi conjecture}, J. Combin. Theory Ser. A
  \textbf{82} (1998), 179--185.

\bibitem{wei}
V.~K. Wei, \emph{A lower bound on the stability number of a simple graph},
  Technical Memorandum TM 81-11217-9, Bell Laboratories (1981).

\end{thebibliography}

\end{document}